\documentclass{article}

\usepackage{filecontents}

\usepackage{amsfonts}
\usepackage{graphicx}
\usepackage{epstopdf}
\usepackage[caption=false]{subfig}
\usepackage{algorithmic}
\usepackage{amssymb}
\usepackage{amsmath}
\usepackage{amsthm}

\newtheorem{theorem}{Theorem}
\newtheorem{prob}{Problem}
\newtheorem{proposition}{Proposition}
\newtheorem{definition}{Definition}

\ifpdf
  \DeclareGraphicsExtensions{.eps,.pdf,.png,.jpg}
\else
  \DeclareGraphicsExtensions{.eps}
\fi

\usepackage{amsopn}

\begin{document}

\title{Minimum-Time Transitions between Thermal and Fixed Average Energy States of the Quantum Parametric Oscillator}

\author{
  Dionisis Stefanatos,\\
  Division of Physical Sciences and Applications,\\
  Hellenic Army Academy,\\
  Vari, Athens 16673,\\
  Greece\\
  \texttt{dionisis@post.harvard.edu}}

\maketitle

\begin{abstract}
In this article we use geometric optimal control to completely solve the problem of minimum-time transitions between thermal equilibrium and fixed average energy states of the quantum parametric oscillator, a system which has been extensively used to model quantum heat engines and refrigerators. We subsequently use the obtained results to find the minimum driving time for a quantum refrigerator and the quantum finite-time availability of the parametric oscillator, i.e. the potential work which can be extracted from this system by a very short finite-time process.
\end{abstract}



\section{Introduction}

\label{sec:intro}

One of the most important reasons for the study of thermodynamics is the development of efficient heat engines and refrigerators. Since modern technology allows the exploitation of tiny length scales, quantum phenomena come into play and determine the behavior of heat machines with small dimensions. As a consequence, studying the properties of quantum heat engines and refrigerators has recently attracted a considerable interest \cite{Feldmann03,Esposito10,Scully11,Abah12,Azimi14,Zhang14,Hardal15,Liu16,Rossnagel16}. In these works, various physical systems are suggested as candidates for the implementation of quantum heat machines, but they all share a common goal: to extract the maximum available work in the minimum possible time. Once the initial and final states which lead to the maximum work extraction have been identified, the problem reduces to finding the minimum-time transition between them. Another important motivation to quickly perform the various steps involved in the thermodynamic cycles of the machines is to reduce the undesirable effects of the environment, which lead to dissipation and decoherence. Several methods have been suggested to speed up quantum heat engines. The simple and robust method of shortcuts to adiabaticity provides a fast interpolation of the path between the initial and the final states \cite{Deng13,Adolfo14,Adolfo16}. Optimal control has been used to obtain the minimum necessary time and the corresponding control which can drive the system between the desired states, under constraints imposed by the experimental setup \cite{Salamon09,Hoffmann13,Stefanatos16a}. And optimization has been exploited in more complex situations, where analytical results are difficult or impossible to find \cite{Stefanatos14,Xiao15,Rezek06}.

The prototype system which has been extensively used in the literature as a model of a quantum heat machine is the quantum parametric oscillator \cite{Salamon09}, a quantum harmonic oscillator whose angular frequency can be altered with time and serves as the control parameter \cite{Rezek06,Abah16}. For this system it was shown in \cite{Salamon09} that, starting from a thermal equilibrium state and changing the frequency from some initial to a lower final value, the maximum work is extracted when the final state of the system is also an equilibrium state. An analytical estimate of the necessary minimum time was also given. In our recent work \cite{Stefanatos16a} we used geometric optimal control and completely solved the problem of minimum-time transitions between thermal equilibrium states of the quantum parametric oscillator, identifying a new type of solution absent from all the previous treatments of the problem \cite{Salamon09,Tsirlin11,Hoffmann13,Boldt16}.

In the present article we study another important problem in the same framework. Specifically, we consider the situation where the initial state of the quantum parametric oscillator is again a thermal equilibrium state, but the final state has now fixed average energy and is not necessarily in thermal equilibrium. Finding the minimum time for this kind of transitions can quantify the so-called \emph{quantum finite-time availability} of the system. This concept describes the potential work which can be extracted from the system by a finite-time process which is too short to gain the maximum available work by bringing the quantum system into thermal equilibrium \cite{Hoffmann15b,Hoffmann15}. The minimum-time solution can also be used to calculate the minimum driving time of a quantum refrigerator, below which the heat machine ceases to operate as a refrigerator. We explain in detail the connection between the control problem and these two important applications from quantum thermodynamics later in the text.

In order to solve the problem of minimum-time transitions we use optimal control theory \cite{Pontryagin,Heinz12}, which has also provided the fastest quantum dynamics in several quantum control applications \cite{Wu02,Boscain05,Boscain06,Bonnard09,Stefanatos11,Bonnard13,Albertini15,Stefanatos16a}.
The paper is organized as follows. In the next section we formulate the problem in terms of optimal control and we subsequently solve it in section \ref{sec:solution}. In section \ref{sec:applications} we explain the connection between this problem and the applications from quantum thermodynamics, while the conclusions follow in
section \ref{sec:conclusions}.

\section{Formulation as an optimal control problem}
\label{sec:formulation}

The system that we consider in this article is a particle of mass $m$ trapped in a parametric harmonic oscillator \cite{Salamon09,Tsirlin11,Hoffmann13,Hoffmann15,Boldt16}. The corresponding Hamiltonian is
\begin{equation}
\label{Hamiltonian}
\hat{H}=\frac{\hat{p}^2}{2m}+\frac{m\omega^2(t)\hat{q}^2}{2},
\end{equation}
where $\hat{q}, \hat{p}$ are the position and momentum operators, respectively, and $\omega(t)$ is the time-varying frequency of the oscillator which serves as the available control and is restricted as
\begin{equation}
\label{frequency_boundary}
\omega(t)=\left\{\begin{array}{cl} \omega_0, & t\leq 0 \\\omega_f, & t\geq T\end{array}\right.
\end{equation}
and
\begin{equation}
\label{frequency}
\omega_f\leq\omega(t)\leq\omega_0,\quad 0<t<T,
\end{equation}
i.e. between its initial and final values $\omega_0,\omega_f$ for the whole time interval with duration $T$.
The time evolution of a quantum observable (hermitian operator) $\hat{O}$ in the Heisenberg picture is given by \cite{Merzbacher98}
\begin{equation}
\label{Observable}
\frac{d\hat{O}}{dt}=\frac{i}{\hbar}[\hat{H},\hat{O}]+\frac{\partial\hat{O}}{\partial t},
\end{equation}
where $i=\sqrt{-1}$ and $\hbar$ is Planck's constant.
The following operators form a closed set under the time evolution generated by $\hat{H}$ \cite{Boldt16}
\begin{equation}
\label{z_operators}
\hat{z}_1=m\hat{q}^2,\quad \hat{z}_2=\frac{\hat{p}^2}{m},\quad \hat{z}_3=-\frac{i}{2\hbar}[\hat{z}_1,\hat{z}_2]=\hat{q}\hat{p}+\hat{p}\hat{q}.
\end{equation}
It is sufficient to follow the expectation values
\begin{equation}
\label{expectations}
z_i=\langle\hat{z}_i\rangle=\mbox{Tr}(\rho_0\hat{z}_i),\quad i=1, 2, 3
\end{equation}
of these operators, where $\rho_0$ is the density matrix corresponding to the initial state of the system at $t=0$ (recall that we use the Heisenberg picture).
From (\ref{Observable}) and (\ref{expectations}) we easily find
\begin{align}
\label{z1}\dot{z}_1  &  = z_3,\\
\label{z2}\dot{z}_2  &  = -\omega^2z_3,\\
\label{z3}\dot{z}_3  &  = -2\omega^2z_1+2z_2.
\end{align}
In order to find the initial values of $z_i$ note that states of thermodynamic equilibrium, with $\omega(t)=\omega$ constant, are characterized by the equipartition of energy $E=\langle\hat{H}\rangle$
\begin{equation}
\label{equipartition}
\left\langle\frac{\hat{p}^2}{2m}\right\rangle=\left\langle\frac{m\omega^2\hat{q}^2}{2}\right\rangle=\frac{E}{2}
\end{equation}
and the absence of correlations
\begin{equation}
\label{no_correlation}
\langle\hat{q}\hat{p}+\hat{p}\hat{q}\rangle=0.
\end{equation}
If the system starts at $t=0$ from the equilibrium state with frequency $\omega_0$ and energy $E_0$, using (\ref{equipartition}) and (\ref{no_correlation}) in (\ref{z_operators}) we find
\begin{equation}
\label{z_initial}
z_1(0)=\frac{E_0}{\omega_0^2},\quad z_2(0)=E_0,\quad z_3(0)=0.
\end{equation}

It can be easily verified that, during the evolution of the system, the following quantity, called the Casimir companion, is a constant of the motion \cite{Boldt13}
\begin{equation}
\label{constant}
z_1z_2-\frac{z_3^2}{4}=\frac{E_0^2}{\omega_0^2}.
\end{equation}
Obviously, it is
\begin{equation}
\label{z_ineq}
z_1z_2\geq\frac{E_0^2}{\omega_0^2}.
\end{equation}
The average energy $E$ of the system corresponding to frequency $\omega$ can be expressed in terms of $z_1,z_2$ as
\begin{equation}
\label{E_z}
E=\left\langle\frac{\hat{p}^2}{2m}+\frac{m\omega^2\hat{q}^2}{2}\right\rangle=\frac{\omega^2}{2}z_1+\frac{1}{2}z_2.
\end{equation}
From (\ref{z_ineq}), (\ref{E_z}) and the fact that the frequency is bounded below by $\omega_f$, we can easily obtain the following minimum value $E_{min}$ \cite{Salamon09}
\begin{equation}
\label{Emin}
E\geq\frac{\omega}{\omega_0}E_0\geq\frac{\omega_f}{\omega_0}E_0=E_{min}.
\end{equation}
For the frequency $\omega(t)$ restricted as in (\ref{frequency}), it has been shown that there is a minimum necessary time $T_{min}$ to achieve the minimum energy $E_{min}$ \cite{Salamon09}. In our recent work \cite{Stefanatos16a} we have thoroughly solved the corresponding optimal control problem, completing thus the previous work on the subject \cite{Salamon09,Tsirlin11,Hoffmann13,Boldt16}. If the available time is less than this minimum time, $T<T_{min}$, then the minimum final energy that can be obtained is always larger than $E_{min}$. In the extreme case $T=0$, the so-called ``sudden quench", where the frequency is  instantaneously reduced from $\omega(0^-)=\omega_0$ to $\omega(0^+)=\omega_f$, the final energy is
\begin{equation}
\label{sudden_quench}
E_{sc}=\frac{\omega^2_f}{2}z_1(0)+\frac{1}{2}z_2(0)=\frac{\gamma^4+1}{2\gamma^4}E_0,\quad\gamma=\sqrt{\frac{\omega_0}{\omega_f}}>1,
\end{equation}
where we have used the initial conditions (\ref{z_initial}) in (\ref{E_z}). The problem that we study in this article is to find the time-varying frequency $\omega(t)$ satisfying (\ref{frequency_boundary}) and (\ref{frequency}) which drives the system from the initial equilibrium state to a state with fixed average energy $E_f$ in the range $E_{min}<E_f<E_{sc}$ in minimum time $T$, where obviously $0<T<T_{min}$. Note that energy values in the range $E_{sc}<E_f<E_{0}$ can be instantaneously obtained with a sudden quench to a final frequency $\omega'_f$ such that $\omega_f<\omega'_f<\omega_0$.

In order to solve this problem, we will use the constant of the motion (\ref{constant}) to reduce the dimension of the system from three to two. Let us define the dimensionless variable $b$ through the relations
\begin{equation}
\label{b}
b=\frac{\sqrt{\langle\hat{q}^2\rangle}}{q_0},\quad q_0=\sqrt{\frac{E_0}{m\omega_0^2}},
\end{equation}
where note that $q_0$ has length dimensions. Then, using the definition of $\hat{z}_1$ from (\ref{z_operators}) and Eqs. (\ref{z1})-(\ref{z3}), variables $z_i$ can be expressed in terms of $b$ as follows
\begin{equation}
\label{zeta}z_1=\frac{E_0}{\omega_0^2}b^2,\quad z_2=\frac{E_0}{\omega_0^2}(b\ddot{b}+\dot{b}^2+\omega^2b^2),\quad z_3=\frac{2E_0}{\omega_0^2}b\dot{b}.
\end{equation}
If we plug (\ref{zeta}) in (\ref{constant}), we obtain the following Ermakov equation for $b(t)$ \cite{Ermakov,Chen10}
\begin{equation}
\label{Ermakov}\ddot{b}(t)+\omega^{2}(t)b(t)=\frac{\omega_{0}^{2}}{b^{3}(t)}%
\end{equation}
If we set
\begin{equation}
\label{x}
x_{1}=b,\quad x_{2}=\frac{\dot{b}}{\omega_{0}},\quad u(t)=\frac{\omega^{2}(t)}%
{\omega_{0}^{2}},
\end{equation}
and rescale time according to $t_{new}=\omega_{0} t_{old}$, we
obtain the following system of first order differential equations, equivalent
to the Ermakov equation 
\begin{align}
\label{system1}\dot{x}_{1}  &  = x_{2},\\
\label{system2}
\dot{x}_{2}  &  = -ux_{1}+\frac{1}{x_{1}^{3}},
\end{align}
where
\begin{equation}
\label{u_bounds}
u_1\leq u(t)\leq u_2,\quad u_1=\frac{\omega_f^2}{\omega_0^2}=\frac{1}{\gamma^4}, \quad u_2=\frac{\omega_0^2}{\omega_0^2}=1.
\end{equation}
In order to find the boundary conditions, we express variables $z_i$ in terms of variables $x_i$ using (\ref{zeta}) and (\ref{x})
\begin{equation}
\label{zeta_x}
z_1=\frac{E_0}{\omega_0^2}x_1^2,\quad z_2=E_0(x_2^2+\frac{1}{x_1^2}),\quad z_3=\frac{2E_0}{\omega_0}x_1x_2,
\end{equation}
where note that we have also used (\ref{Ermakov}) to replace the second derivative of $b$ in (\ref{zeta}). Using (\ref{z_initial}) and (\ref{zeta_x}) we obtain the initial conditions
\begin{equation}
\label{x_initial}
x_1(0)=1,\quad x_2(0)=0.
\end{equation}
The energy at the final point $(x_1(T),x_2(T))$, where the frequency is $\omega=\omega_f$, is set to $E=E_f$. Using (\ref{zeta_x}) in (\ref{E_z}), we find that the coordinates of the final point should belong to the following curve
\begin{equation}
\label{final_curve}
x_2^2(T)+u_1x_1^2(T)+\frac{1}{x_1^2(T)}=\frac{2}{r_E},\quad r_E=\frac{E_0}{E_f},
\end{equation}
where note that, since $E_{min}<E_f<E_{sc}$, the energy ratio $r_E$ is in the range
\begin{equation}
\label{r}
\frac{2\gamma^2}{\gamma^4+1}<r_E<\gamma^2,\quad\gamma=\sqrt{\frac{\omega_0}{\omega_f}},
\end{equation}
as derived from (\ref{Emin}) and (\ref{sudden_quench}). We end up with the following optimal control problem for system (\ref{system1}), (\ref{system2}):
\begin{prob}\label{problem1}
Find $u_1\leq u(t)\leq u_2$ with $u(0)=u_2=1, u(T)=u_1=1/\gamma^4$, such that starting from $(x_1(0),x_2(0))=(1,0)$ the above system reaches the final curve (\ref{final_curve}) in minimum time $T$.
\end{prob}

In the next section we solve the following optimal control problem, where we drop the boundary conditions on the control $u$ corresponding to the frequency boundary conditions (\ref{frequency_boundary}), as we did in our previous work \cite{Stefanatos11} and justify below:

\begin{prob}\label{problem2}
Find $u_1\leq u(t)\leq u_2$, with $u_1=1/\gamma^4,u_2=1$, such that starting from $(x_1(0),x_2(0))=(1,0)$ the system above reaches the final curve (\ref{final_curve}) in minimum time $T$.
\end{prob}

In both problems the class of admissible controls formally are Lebesgue measurable functions which take values in the control set $[u_1,u_2]$ almost everywhere. However, as we shall see, optimal controls are piecewise continuous, in fact bang-bang. The optimal control found for Problem \ref{problem2} is also optimal for Problem \ref{problem1}, with the addition of instantaneous jumps at the initial and final points, so that the boundary conditions $u(0)=1$ and $u(T)=1/\gamma^4$ are satisfied. Note that in connection with (\ref{frequency_boundary}), a natural way to think about these conditions is that $u(t)=1$ for $t\leq 0$ and $u(t)=1/\gamma^4$ for $t\geq T$; in the interval $(0,T)$ we pick the control that achieves the desired transfer in minimum time.


\section{Optimal solution}
\label{sec:solution}

In our recent work \cite{Stefanatos16a} we solved the following problem, where the final point was fixed on the $x_1$-axis.
\begin{prob}\label{problem3}
Find $u_1\leq u(t)\leq u_2$, with $u_1=1/\gamma^4,u_2=1$, such that starting from $(x_1(0),x_2(0))=(1,0)$, the system above reaches the final point $(x_1(T),x_2(T))=(\gamma,0), \gamma>1$, in minimum time $T$.
\end{prob}
Obviously, this problem is closely related to Problem \ref{problem2}, and in this section we investigate how our previous solution is modified due to the requirement that the final point now belongs to the curve (\ref{final_curve}).

The system described by (\ref{system1}), (\ref{system2}) can be expressed
in compact form as
\begin{equation}
\dot{x}=f(x)+ug(x), \label{affine}%
\end{equation}
where the vector fields are given by
\begin{equation}
f=\begin{pmatrix}x_{2}\\1/x_{1}^{3}\end{pmatrix},
\quad
g=\begin{pmatrix}0\\-x_{1}\end{pmatrix}
\end{equation}
and $x\in\mathcal{D}=\{(x_{1},x_{2})\in\mathbb{R}^{2}:x_{1}>0\}$, $u\in
U=[u_{1},u_{2}]$. Admissible controls are Lebesgue measurable functions that
take values in the control set $U$. Given an admissible control $u$ defined
over an interval $[0,T]$, the solution $x$ of the system (\ref{affine})
corresponding to the control $u$ is called the corresponding trajectory and we
call the pair $(x,u)$ a controlled trajectory. Note that the domain
$\mathcal{D}$\ is invariant in the sense that trajectories cannot leave
$\mathcal{D}$. Starting with any positive initial condition $x_{1}(0)>0$ and
using any admissible control $u$, as $x_{1}\rightarrow0^{+}$ the
``repulsive force" $1/x_{1}^{3}$ leads to an increase in
$x_{1}$ that will keep $x_{1}$ positive (as long as the solutions exist).

For a constant $\lambda_{0}$ and a row vector $\lambda=(\lambda_{1}%
,\lambda_{2})\in\left(  \mathbb{R}^{2}\right)  ^{\ast}$ define the control
Hamiltonian as%
\[
H=H(\lambda_{0},\lambda,x,u)=\lambda_{0}+\langle\lambda,f(x)+ug(x)\rangle.
\]
Pontryagin's Maximum Principle for \emph{time-optimal} processes \cite{Pontryagin}
provides the following necessary conditions for optimality, which hold for both Problems \ref{problem2} and \ref{problem3}, although the final point for the former is unspecified:

\begin{theorem}[Maximum principle for time-optimal processes]\label{prop:max_principle}
Let $(x_{\ast}(t),u_{\ast}(t))$
be a time-optimal controlled trajectory that transfers the initial condition
$x(0)=x_{0}$ into the terminal state $x(T)=x_T$. Then it is a necessary
condition for optimality that there exists a constant $\lambda_{0}\leq0$ and
nonzero, absolutely continuous row vector function $\lambda(t)$ such that:

\begin{enumerate}
\item $\lambda$ satisfies the so-called adjoint equation%
\[
\dot{\lambda}(t)=-\frac{\partial H}{\partial x}(\lambda_{0},\lambda
(t),x_{\ast}(t),u_{\ast}(t))
\]

\item For $0\leq t\leq T$ the function $u\mapsto H(\lambda_{0}%
,\lambda(t),x_{\ast}(t),u)$ attains its maximum\ over the control set $U$ at
$u=u_{\ast}(t)$.

\item $H(\lambda_{0},\lambda(t),x_{\ast}(t),u_{\ast}(t))\equiv0$.
\end{enumerate}
\end{theorem}

We call a controlled trajectory $(x,u)$ for which there exist multipliers
$\lambda_{0}$ and $\lambda(t)$ such that these conditions are satisfied an
extremal. Extremals for which $\lambda_{0}=0$ are called abnormal. If
$\lambda_{0}<0$, then without loss of generality we may rescale the $\lambda
$'s and set $\lambda_{0}=-1$. Such an extremal is called normal.

For the system (\ref{system1}), (\ref{system2}) we have
\begin{equation*}
H(\lambda_{0},\lambda,x,u)=\lambda_{0}+\lambda_{1}x_{2}+\lambda_{2}\left(  \frac{1}%
{x_{1}^{3}}-x_{1}u\right)  ,
\end{equation*}
and thus
\begin{equation}\label{adjoint}
\dot{\lambda}=-\lambda\begin{pmatrix}0&1\\-(u+3/x_{1}^{4})&0\end{pmatrix}=-\lambda A
\end{equation}
Observe that $H$ is a linear function of the bounded control variable $u$. The
coefficient at $u$ in $H$ is $-\lambda_{2}x_{1}$ and, since $x_{1}>0$, its
sign is determined by $\Phi=-\lambda_{2}$, the so-called \emph{switching
function}. According to the maximum principle, point 2 above, the optimal
control is given by $u=u_{1}$ if $\Phi<0$ and by $u=u_{2}$ if $\Phi>0$. The
maximum principle provides a priori no information about the control at times
$t$ when the switching function $\Phi$ vanishes. However, if $\Phi(t)=0$ and
$\dot{\Phi}(t)\neq0$, then at time $t$ the control switches between its
boundary values and we call this a bang-bang switch. If $\Phi$ were to vanish
identically over some open time interval $I$ the corresponding control is
called \emph{singular}.

\begin{proposition}
Optimal controls are bang-bang and all the extremals are normal.
\end{proposition}

\begin{proof}
Analogous to the proof of Propositions 1 and 2 in \cite{Stefanatos16a}.
\end{proof}


\begin{definition}
We denote the vector fields corresponding to the constant bang controls
$u_{1}$ and $u_{2}$ by $X=f+u_{1}g$ and $Y=f+u_{2}g$, respectively, and call
the trajectories corresponding to the constant controls $u\equiv u_{1}$ and
$u\equiv u_{2}$ $X$- and $Y$-trajectories. A concatenation of an
$X$-trajectory followed by a $Y$-trajectory is denoted by $XY$ while the
concatenation in the inverse order is denoted by $YX$.
\end{definition}

For normal extremals we can set $\lambda_{0}=-1$.
Then, $H=0$ implies that for any switching time $t_0$, where $\lambda_2(t_0)=-\Phi(t_0)=0$, we must have $\lambda
_{1}(t_0)x_{2}(t_0)=1$. For an $XY$ junction we have $\dot{\Phi}(t_0)=-\dot{\lambda}_2(t_0)=\lambda_{1}(t_0)>0$ and thus necessarily $x_{2}(t_0)>0$. Analogously,
optimal $YX$ junctions need to lie in $\{x_{2}<0\}$. 

In the following proposition we summarize some additional facts about the solution of Problem \ref{problem2}, obtained from the solution of Problem \ref{problem3} in \cite{Stefanatos16a}.

\begin{proposition}\label{summary}
The extremal trajectories have the form $XYX\ldots XY$, with an odd number of switchings. The ratio of the coordinates $(x_2/x_1)$ of consecutive switching points has constant magnitude but alternating sign, while these points are not symmetric with respect to the $x_{1}$-axis. If $s=x_2^2/x_1^2$ is the square of this ratio, which is obviously constant at the switching points, then the times spent on each intermediate $X$ and $Y$ segments of the trajectory (excluding the first $X$ and the last $Y$ segments) are
\begin{align}
\label{switch_X}
\tau_{X}  &  = \frac{1}{2\sqrt{u_{1}}}\cos^{-1}\left(
\frac{s-u_{1}}{s+u_{1}}\right),\\
\label{switch_Y}
\tau_{Y}  &  = \frac{1}{2\sqrt{u_{2}}}\left[  2\pi-\cos^{-1}\left(  \frac
{s-u_{2}}{s+u_{2}}\right)  \right].
\end{align}
\end{proposition}

\begin{proof}
Note that the optimal trajectory cannot start with a $Y$-segment, since for $u=u_2=1$ the initial point $(1,0)$ is an equilibrium point for system (\ref{system1}),(\ref{system2}). Also, the optimal trajectory cannot end with an $X$-segment, since these segments have the form $x_2^2+u_1x_1^2+1/x_1^2=c$, where $c$ constant, and they do not intersect the final curve (\ref{final_curve}), which has a similar form, when $c\neq 2/r_E$. Consequently, the extremal trajectories should start with an $X$-segment and end with a $Y$-segment. From this and the bang-bang form of the optimal control we conclude that the extremal trajectories have the form $XYX\ldots XY$, with an odd number of switchings. The property of the coordinates ratio at the switching points is proved in Lemma 3 in \cite{Stefanatos16a}, while the times $\tau_X, \tau_Y$ as functions of the ratio $s$ are taken from Theorem 2 in \cite{Stefanatos16a}.
\end{proof}

Up to now we have presented the characheristics of the optimal solution which are common in both Problems \ref{problem2} and \ref{problem3}. But the adjoint vector $\lambda$ for Problem \ref{problem2} should additionally satisfy the transversality conditions at the final time $t=T$, which state that the vector $\lambda(T)$ should be orthogonal to the tangent vector of the curve (\ref{final_curve}) at the final point \cite{Pontryagin}. In the following proposition, which is the main technical point in this paper, we use the transversality conditions to express the time spent on the final $Y$-segment as a function of the ratio $s$.

\begin{proposition}\label{transvers}
Let $P=(x_1,x_2)$ be the last switching point, $s=x_2^2/x_1^2$ the ratio of the squares of the coordinates, and $\tau$ the time to reach from $P$ the final point $F$ on the curve (\ref{final_curve}). Then:
\begin{align}
\label{cosine_tau}
\cos(2\sqrt{u_{2}}\tau)  &  = \frac{-s(u_1+u_2)+u_{2}%
\sqrt{(u_2-u_1)^{2}-4su_1}}{(s+u_{2})(u_2-u_1)},\\
\label{sine_tau}
\sin(2\sqrt{u_{2}}\tau)  &  = \frac{\sqrt{su_2}\left[u_1+u_2+
\sqrt{(u_2-u_1)^{2}-4su_1}\right]}{(s+u_{2})(u_2-u_1)}.
\end{align}
\end{proposition}

\begin{proof}
The method that we will use to obtain the above formulas is similar to the one we used in \cite{Stefanatos11,Stefanatos16a}
to obtain the interswitching times (\ref{switch_X}) and (\ref{switch_Y}), which
was based on the concept of ``conjugate point" for bang-bang controls \cite{Suplane,Boscain04}.
Without loss of generality assume that the trajectory passes through $P$ at
time $0$ and is at $F$ at time $\tau$. First of all note that, since $P$ is a switching point,
the corresponding multiplier vanishes against the control vector field $g$ at
this point, i.e., $\langle\lambda(0),g(P)\rangle=0$. Next, observe that the tangent vector to the final curve (\ref{final_curve}) coincides with the vector field $X=f+u_1g$ evaluated at the points of the curve. According to the transversality conditions, the multiplier at the final time should vanish against this vector field at the final point $F$, i.e. $\langle\lambda(\tau),X(F)\rangle=0$. We need to compute what this last relation implies at time $0$.
In order to do so, we move the vector $X(F)$ along the last $Y$-segment backward from $F$ to $P$. This is
done by means of the solution $w(t)$ of the variational equation along the
$Y$-trajectory with terminal condition $w(\tau)=X(F)$ at time $\tau$. Recall
that the variational equation along $Y$ is the linear system $\dot{w}=Aw$
where matrix $A$ is given in (\ref{adjoint}). Symbolically, if we denote by
$e^{tY}(P)$ the value of the $Y$-trajectory at time $t$ that starts at
point $P$ at time $0$ and by $(e^{-tY})_{\ast}$ the backward evolution under
the linear differential equation $\dot{w}=Aw$, then we can represent this
solution in the form
\begin{align}
w(0)&=(e^{-\tau Y})_{\ast}w(\tau)=(e^{-\tau Y})_{\ast}X(F)\nonumber\\
&=(e^{-\tau Y})_{\ast}X(e^{\tau Y}(P))=(e^{-\tau Y})_{\ast}\circ X\circ e^{\tau Y}(P).\nonumber
\end{align}
Since the
\textquotedblleft adjoint equation\textquotedblright\ of the maximum principle
is precisely the adjoint equation to the variational equation, it follows that
the function $t\mapsto\langle\lambda(t),w(t)\rangle$ is constant along the
$Y$-trajectory. Hence $\langle\lambda(\tau),X(F)\rangle=0$ implies that
\[
\langle\lambda(0),w(0)\rangle=\langle\lambda(0),(e^{-\tau Y})_{\ast}X(e^{\tau
Y}(P))\rangle=0
\]
as well. But the non-zero two-dimensional multiplier $\lambda(0)$ can only be orthogonal to
both $g(P)$ and $w(0)$ if these vectors are parallel, $g(P)\Vert
w(0)=(e^{-\tau Y})_{\ast}X(e^{\tau Y}(P))$. It is this relation that defines
the time $\tau$ spent on the last $Y$-segment.

It remains to compute $w(0)$. For this we make use of the well-known relation \cite{Heinz12}
\begin{equation*}
(e^{-\tau Y})_{\ast}\circ X\circ e^{\tau Y}=e^{\tau\,adY}(X)
\end{equation*}
where the operator $adY$ is defined as $adY(X)=[Y,X]$, with $[,]$ denoting the
Lie bracket of the vector fields $Y$ and $X$. 
For our system, the Lie algebra 
generated by the fields $f$ and $g$ actually is finite dimensional: we have
\[
\lbrack f,g](x)=\left(
\begin{array}
[c]{c}%
x_{1}\\
-x_{2}%
\end{array}
\right)
\]
and the relations
\[
\lbrack f,[f,g]]=2f,\qquad\lbrack g,[f,g]]=-2g
\]
can be directly verified. Using these relations and the analyticity of the
system, $e^{t\,adY}(X)$ can be calculated in closed form from the expansion
\begin{equation*}
e^{t\,adY}(X)=\sum_{n=0}^{\infty}\frac{t^{n}}{n!}\,ad\,^{n}Y(X),
\end{equation*}
where, inductively, $ad^{n}Y(X)=[Y,ad^{n-1}Y(X)]$.
It is not hard to show that for $n=0,1,2,\ldots$, we have that
\[
ad\,^{2n+1}Y(X)=-(u_2-u_1)(-4u_{2})^{n}[f,g]
\]
and
\[
ad\,^{2n+2}Y(X)=-2(u_2-u_1)(-4u_{2})^{n}(f-u_{2}g),
\]
so that
\begin{multline*}
e^{t\,adY}(X)=\\X-(u_2-u_1)\left\{\sum_{n=0}^{\infty}\frac{t^{2n+1}}{(2n+1)!}\,(-4u_{2}%
)^{n}[f,g]+\sum_{n=0}^{\infty}\frac{2t^{2n+2}}{(2n+2)!}\,(-4u_{2})^{n}(f-u_{2}g)\right\}.
\end{multline*}
By summing the series appropriately we obtain
\begin{multline*}
e^{t\,adY}(X)=\\X-(u_2-u_1)\left\{\frac{1}{2\sqrt{u_{2}}}\sin(2\sqrt{u_{2}}t)[f,g]+\frac
{1}{2u_{2}}[1-\cos(2\sqrt{u_{2}}t)](f-u_{2}g)\right\}.
\end{multline*}
The field $w(0)=(e^{-\tau Y})_{\ast}X(e^{\tau Y}(P))=e^{t\,adY}(X(P))$ is
parallel to $g(P)=(0,-x_{1})^{T}$ if and only if
\begin{equation*}
x_2-(u_2-u_1)\left\{\frac{x_1}{2\sqrt{u_{2}}}\sin(2\sqrt{u_{2}}\tau)+\frac{x_{2}}{2u_2}\left[  1-\cos(2\sqrt{u_{2}}\tau)\right]\right\}=0.
\end{equation*}
Hence
\begin{equation*}
\sin(2\sqrt{u_{2}}\tau)=\frac{x_{2}}{\sqrt{u_{2}}x_{1}}\left[\frac{u_2+u_1}{u_2-u_1}+\cos(2\sqrt{u_{2}}\tau)\right]=\sqrt{\frac{s}{u_2}}\left[\frac{u_2+u_1}{u_2-u_1}+\cos(2\sqrt{u_{2}}\tau)\right],
\end{equation*}
where the last equality follows from the fact that for the last switching point it is $x_2>0$, thus $x_2/x_1=\sqrt{s}$.
Using the above equation, the expressions (\ref{cosine_tau}) and (\ref{sine_tau}) can be easily obtained. 
\end{proof}

In the following theorem, we combine Propositions \ref{summary} and \ref{transvers} to obtain a transcendental equation for the ratio $s$ and an expression for the total time to reach the final curve along the extremal trajectories.

\begin{theorem}
\label{prop:time}
The extremal trajectories have the form $XYX\ldots XY$, with an odd number of switchings. The necessary time to reach the
target curve (\ref{final_curve}) with $2n+1$ switchings, $n=0,1,2,\ldots$, is
\begin{equation}
\label{time_odd}
T^{\pm}_{2n+1}=\tau_{\pm}+n(\tau_{X}+\tau_{Y})+\tau,
\end{equation}
where
\begin{align}
\label{time_in}\tau_{\pm}  &  = \frac{1}{2\sqrt{u_{1}}}\cos^{-1}\left[
\frac{sc_{1}\mp u_{1}\sqrt{c_{1}^{2}-4(s+u_{1})}}{(s+u_{1})\sqrt{c_{1}%
^{2}-4u_{1}}}\right],\\
\label{time_fi}
\tau  &  = \frac{1}{2\sqrt{u_{2}}}\cos^{-1}\left[ \frac{-s(u_1+u_2)+u_{2}%
\sqrt{(u_2-u_1)^{2}-4su_1}}{(s+u_{2})(u_2-u_1)} \right]  ,
\end{align}
the interswitching times $\tau_X, \tau_Y$ are given in (\ref{switch_X}), (\ref{switch_Y}), respectively, while the constant
\begin{equation}\label{c1}
c_1=1+u_1
\end{equation}
characterizes the first $X$-segment of the trajectory. The ratio $s$ of the square of coordinates at the switching points is the solution in the interval $0<s\leq (1-u_{1})^{2}/4$ of the following transcendental equation
\begin{equation}
\label{transcendental}
y(\kappa_{\pm})=y(\kappa_{2n+1,\pm})\cos(2\sqrt{u_{2}}\tau)+\sqrt{1-y^2(\kappa_{2n+1,\pm})}\sin(2\sqrt{u_{2}}\tau),
\end{equation}
where $\cos(2\sqrt{u_{2}}\tau),\sin(2\sqrt{u_{2}}\tau)$ are given in (\ref{cosine_tau}), (\ref{sine_tau}) as functions of $s$, while
\begin{equation}\label{y}
y(\kappa_{2n+1,\pm})=\frac{2u_2\kappa^2_{2n+1,\pm}-c_{\pm}}{\sqrt{c^2_{\pm}-4u_2}},\quad y(\kappa_{\pm})=\frac{2u_2\kappa^2_{\pm}-c_{\pm}}{\sqrt{c^2_{\pm}-4u_2}},
\end{equation}
\begin{equation}\label{kappa}
\kappa^2_{2n+1,\pm}=\left(\frac{s+u_2}{s+u_1}\right)^n\frac{c_1\pm\sqrt{c_1^2-4(s+u_1)}}{2(s+u_1)},\quad \kappa^2_{\pm}=\frac{c_{\pm}r_E-2}{r_E(u_2-u_1)},
\end{equation}
and
\begin{equation}\label{c_const}
c_{\pm}=(s+u_2)\kappa^2_{2n+1,\pm}+\frac{1}{\kappa^2_{2n+1,\pm}}.
\end{equation}
The ratio between the initial and final energies $r_E=E_0/E_f$ characterizes the final curve (\ref{final_curve}).
Note that the $\pm$ sign in (\ref{transcendental}) corresponds to the $\pm$ sign in (\ref{time_odd}).
\end{theorem}

\begin{proof}
Consider an extremal trajectory $XYX\ldots XY$ with $2n+1$ switching points $P_j(\kappa_{j},\mu_{j}), j=1,2,\ldots,2n+1$, and final point $F(\kappa,\mu)$ on the curve (\ref{final_curve}), shown in Fig. \ref{fig:proof}. The $X$-segments of the trajectory ($u=u_1$) are depicted with blue solid line and the $Y$-segments ($u=u_2$) with red dashed line, while the final curve is the black solid line where the trajectory terminates. Observe that the odd-numbered switching points lie on a positive-slope straight line passing through the origin, while the even-numbered switching points lie on the symmetric line with opposite slope, in accordance with Proposition \ref{summary}. We will express in terms of the ratio $s$ the $x_1$-coordinates of the last switching point $P_{2n+1}(\kappa_{2n+1},\mu_{2n+1})$ and of the final point $F(\kappa,\mu)$, and then we will connect them by integrating the equations of motion along the last $Y$-segment for time $\tau$ given in Proposition \ref{transvers}.
\begin{figure}[htbp]
  \centering
  \includegraphics[width=0.5\linewidth]{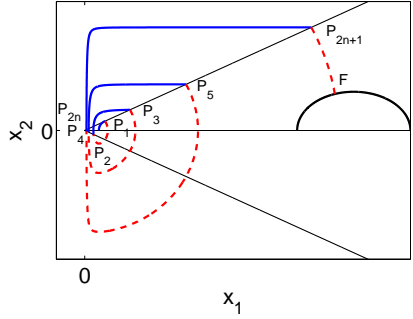}
  \caption{Extremal trajectory $XYX\ldots XY$ with $2n+1$ switchings. Blue solid line corresponds to $X$-segments ($u=u_1$) and red dashed line corresponds to $Y$-segments ($u=u_2$).}
   \label{fig:proof}
\end{figure}

Two consecutive switching points satisfy the following equation
\begin{equation}
\label{connection}
\mu_{j+1}^{2}+u\kappa_{j+1}^{2}+\frac{1}{\kappa_{j+1}^{2}}=\mu_{j}^{2}+u\kappa_{j}^{2}+\frac{1}{\kappa_{j}^{2}},
\end{equation}
where $u=u_1$ if the two points are connected with an $X$-segment and $u=u_2$ if they are joined with a $Y$-segment (it can be verified from the system equations that the quantity $x_2^2+ux_1^2+1/x_1^2$ is constant along segments with constant control $u$).
The \emph{ratio of the squares} of the coordinates of all the switching points is constant and equal to $s$, thus $\mu_{j+1}^{2}/\kappa_{j+1}^{2}=\mu_{j}^{2}/\kappa_{j}^{2}=s$ and (\ref{connection}) becomes
\begin{equation*}
(\kappa_{j+1}^{2}-\kappa_{j}^{2})\left(s+u-\frac{1}{\kappa_{j}^{2}\kappa_{j+1}^{2}}\right)=0.
\end{equation*}
But $\kappa_{j+1}\neq\kappa_{j}$ since the consecutive switching points are not symmetric with respect to $x_1$-axis (Proposition \ref{summary}), thus
\begin{equation}
\label{consecutive}
\kappa_{j+1}^{2}=\frac{1}{\kappa_{j}^{2}(s+u)}.
\end{equation}
If we consecutively apply (\ref{consecutive}) from the first switching point up to the last, we obtain
\begin{equation}
\label{X_first_last}
\frac{\kappa_{2n+1}^2}{\kappa_{1}^2}=\left(\frac{s+u_2}{s+u_1}\right)^n.
\end{equation}
Since the first switching point $P_1(\kappa_1,\mu_1)$ belongs to the first $X$-segment starting from $(1,0)$, it satisfies the equation
\begin{equation*}
\mu_{1}^{2}+u_1\kappa_{1}^{2}+\frac{1}{\kappa_{1}^{2}}=c_1\Rightarrow(s+u_1)\kappa_1^4-c_1\kappa_1^2+1=0,
\end{equation*}
where $c_1=u_1+1$. Solving for $\kappa_1^2$ we obtain
\begin{equation}
\label{k1}
\kappa_{1,\pm}^{2}=\frac{c_{1}\pm\sqrt{c_{1}^{2}-4(s+u_{1})}}{2(s+u_{1})}.
\end{equation}
If we plug (\ref{k1}) in (\ref{X_first_last}), we find the expression (\ref{kappa}) of $\kappa^2_{2n+1,\pm}$ in terms of the ratio $s$, where $\pm$ correspond to the $\pm$ sign in (\ref{k1}).

We next move to find the expression in terms of $s$ for the $x_1$-coordinate of the final point $F(\kappa,\mu)$. This point belongs to the last $Y$-segment starting from $P_{2n+1}(\kappa_{2n+1},\mu_{2n+1})$, thus its coordinates satisfy the following equation
\begin{align}\label{Y_last}
\mu^{2}+u_2\kappa^{2}+\frac{1}{\kappa^{2}}
&=\mu_{2n+1,\pm}^{2}+u_2\kappa_{2n+1,\pm}^{2}+\frac{1}{\kappa_{2n+1,\pm}^{2}}\nonumber\\
&=(s+u_2)\kappa_{2n+1,\pm}^{2}+\frac{1}{\kappa_{2n+1,\pm}^{2}}\nonumber\\
&=c_{\pm},
\end{align}
where $c_{\pm}$ is defined in (\ref{c_const}).
But $F$ is also a point of the final curve (\ref{final_curve}), thus
\begin{equation}
\label{final_point}
\mu^{2}+u_1\kappa^{2}+\frac{1}{\kappa^{2}}=\frac{2}{r_E}.
\end{equation}
By subtracting (\ref{final_point}) from (\ref{Y_last}), we easily obtain the expression (\ref{kappa}) for $\kappa^2_{\pm}$, where $\pm$ correspond to the $\pm$ sign in (\ref{k1}).

Now we are in a position to connect points $P_{2n+1},F$ by integrating the system equations along the last $Y$-segment of the trajectory. The points $(x_1,x_2)$ of this segment satisfy the equation
\begin{equation}
x_2^2+u_2x_1^2+\frac{1}{x_1^2}=c_{\pm},\nonumber
\end{equation}
where $c_{\pm}$ is given in (\ref{c_const}). The last $Y$-segment lies on the upper quadrant $x_2>0$, thus, from the above equation and (\ref{system1}) we have
\begin{equation}
\dot{x}_1=x_{2}=\frac{\sqrt{-u_{2}x_{1}^{4}+c_{\pm}x_{1}^{2}-1}}{x_{1}}.\nonumber
\end{equation}
If we make the change of variables
\begin{equation}
y=\frac{2u_{2}x_{1}^{2}-c_{\pm}}{\sqrt{c^{2}_{\pm}-4u_{2}}}\nonumber
\end{equation}
we obtain
\begin{equation}
-\frac{dy}{\sqrt{y^{2}-1}}=-2\sqrt{u_{2}}dt.\nonumber
\end{equation}
By integrating the last equation from $t=0$ (point $P_{2n+1}$) to $t=\tau$ (point $F$), where $\tau$ is the time from Proposition \ref{transvers} given in (\ref{time_fi}), we find
\begin{equation}
\label{int_final}
\cos^{-1}y(\kappa_{\pm})-\cos^{-1}y(\kappa_{2n+1,\pm})=-2\sqrt{u_{2}}\tau,
\end{equation}
where $y(\kappa_{\pm}),y(\kappa_{2n+1,\pm})$ are given in (\ref{y}). If we move $\cos^{-1}y(\kappa_{2n+1,\pm})$ to the right and then take the $\cos$ of both sides, we obtain (\ref{transcendental}). Note that this is a transcendental equation for the ratio $s\geq 0$, since all the terms involved are expressed as functions of this ratio. In order to find the range of $s$, we require the nonnegativity of the quantities under the square roots in (\ref{kappa}), (\ref{time_fi}) and we obtain $0<s\leq\mbox{Min}\{(1-u_{1})^{2}/4, (u_{2}-u_1)^{2}/4u_1\}$, where the value $s=0$ is excluded since the switching points do not lie on the $x_1$-axis. But $u_2=1$ and $u_1=1/\gamma^4<1$, thus $(u_{2}-u_1)^{2}/4u_1=(1-u_1)^{2}/4u_1\geq (1-u_{1})^{2}/4$ and finally $0<s\leq (1-u_{1})^{2}/4$.

Having found the ratio $s$, it is not hard to find the duration of the trajectory. An extremal with $2n+1$ switchings contains $n$ ``turns", where each ``turn" consists of one $X$ and one $Y$ intermediate segments, with durations $\tau_X,\tau_Y$ given in (\ref{switch_X}), (\ref{switch_Y}), respectively. The total duration of the trajectory is given by (\ref{time_odd}), where $\tau_{\pm}$ is the time spent on the first $X$-segment and $\tau$ is the time spent on the final $Y$-segment. From (\ref{cosine_tau}) we can easily obtain the expression (\ref{time_fi}) for $\tau$. It remains to calculate $\tau_{\pm}$. If we integrate the equations of motion from the starting point $(1,0)$ to the first switching point $P_{1}(\kappa_{1,\pm},\mu_{1,\pm})$ we find, similarly to (\ref{int_final})
\begin{equation}
\label{int_initial}
\cos^{-1}\tilde{y}(\kappa_{1,\pm})-\cos^{-1}\tilde{y}(1)=-2\sqrt{u_{1}}\tau_{\pm},
\end{equation}
where
\begin{align}
\label{first_switch}\tilde{y}(\kappa_{1,\pm})&=\frac{2u_{1}\kappa_{1,\pm}^{2}-c_{1}}{\sqrt{c^{2}_{1}-4u_{1}}}=\frac{-sc_{1}\pm u_{1}\sqrt{c_{1}^{2}-4(s+u_{1})}}{(s+u_{1})\sqrt{c_{1}^{2}-4u_{1}}},\\
\label{start}\tilde{y}(1)&=\frac{2u_{1}-c_{1}}{\sqrt{c^{2}_{1}-4u_{1}}}=-1.
\end{align}
Note that in the last two equations we used the expressions (\ref{k1}) for $\kappa_{1,\pm}^2$ and (\ref{c1}) for $c_1$. From (\ref{start}) we have $\cos^{-1}\tilde{y}(1)=\pi$, thus (\ref{int_initial}) becomes
\begin{equation*}
\cos^{-1}\tilde{y}(\kappa_{1,\pm})=\pi-2\sqrt{u_{1}}\tau_{\pm}\Rightarrow \tilde{y}(\kappa_{1,\pm})=-\cos2\sqrt{u_1}\tau_{\pm}.
\end{equation*}
From the last relation and (\ref{first_switch}) we obtain the expression (\ref{time_in}) for $\tau_{\pm}$.
\end{proof}

We can test the transcendental equation (\ref{transcendental}) of Theorem \ref{prop:time} by examining the limiting case $r_E\rightarrow\gamma^2$ $(E_f\rightarrow E_{min})$, where the final curve (\ref{final_curve}) shrinks to the point $(\gamma,0)$ on the $x_1$-axis. Instead of testing directly (\ref{transcendental}), it is actually easier to check (\ref{int_final}), from which the transcendental equation is derived. Since the final point is now $(\gamma,0)$, it is $\kappa_{\pm}=\gamma$ and the constant characterizing the last $Y$-segment becomes
\begin{equation*}
c=c_{\pm}=u_2\gamma^2+\frac{1}{\gamma^2}.
\end{equation*}
Using this expression in (\ref{y}) we find $y(\gamma)=1$, thus $\cos^{-1}y(\gamma)=0$ and (\ref{int_final}) becomes
\begin{equation*}
y(\kappa_{2n+1,\pm})=\cos(2\sqrt{u_{2}}\tau)=\frac{-sc+u_2\sqrt{c^2-4(s+u_2)}}{(s+u_2)\sqrt{c^2-4u_2}},
\end{equation*}
where the last equation for $\cos(2\sqrt{u_{2}}\tau)$ is obtained from (\ref{cosine_tau}) by making the replacement $u_1=1/\gamma^4$ and then multiplying the numerator and the denominator with $\gamma^2$. If we use in the above equation the expression for $y(\kappa_{2n+1,\pm})$ from (\ref{y}) we find
\begin{equation*}
\kappa_{2n+1,\pm}^2=\frac{c+\sqrt{c^2-4(s+u_2)}}{2(s+u_2)},
\end{equation*}
and if we replace $\kappa_{2n+1,\pm}^2$ with the corresponding expression from (\ref{kappa}), we end up with the transcendental equation
\begin{equation*}
\frac{c+\sqrt{c^{2}-4(s+u_{2})}}{c_{1}\pm\sqrt{c_{1}^{2}-4(s+u_{1})}}=\left(\frac{s+u_{2}}{s+u_{1}}\right)^{n+1}.
\end{equation*}
This is exactly the transcendental equation obtained in \cite{Stefanatos16a} where we solved Problem \ref{problem3}, with the final point fixed on the $x_1$-axis.

\section{Applications in quantum thermodynamics}
\label{sec:applications}

\subsection{Calculation of the minimum driving time for a quantum refrigerator}

We consider a quantum refrigerator based on a parametric harmonic oscillator which is studied in \cite{Abah16}. The frequency of the oscillator determines the spatial extent of the wavefunctions and thus it is analogous to the inverse volume in the classical setting. A frequency increase corresponds to a compression, while a frequency decrease to an expansion. The refrigerator operates between a cold and a hot reservoir and, as its classical counterpart, it consumes work to extract heat from the cold reservoir. In order to achieve this it executes the Otto cycle, consisting of four branches which are depicted in Fig. \ref{fig:refrigerator}: (1) \emph{Isentropic compression} $A\rightarrow B$: initially (state $A$) the oscillator is in thermal equilibrium with the cold reservoir at temperature $\mathcal{T}_c$, with its frequency fixed to the value $\omega_c$. Then, it is isolated from the reservoir and its frequency is increased to $\omega_h$. During this process, work is added to the system while the entropy remains constant. (2) \emph{Hot isochore} $B\rightarrow C$: The frequency is kept fixed to $\omega_h$ while the oscillator is coupled to the hot reservoir and reaches a thermal equilibrium state $C$ at temperature $\mathcal{T}_h$. (3) \emph{Isentropic expansion} $C\rightarrow D$: the frequency is decreased back to the initial value $\omega_c$ at constant entropy. (4) \emph{Cold isochore} $D\rightarrow A$: the system is brought to contact with the cold reservoir and returns to the initial thermal equilibrium state $A$ at temperature $\mathcal{T}_c$.
\begin{figure}[htbp]
  \centering
  \includegraphics[width=0.5\linewidth]{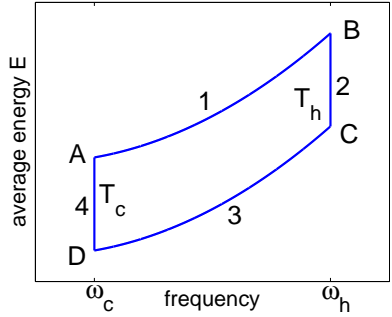}
  \caption{Energy-frequency diagram of a quantum refrigerator executing the Otto cycle.}
  \label{fig:refrigerator}
\end{figure}

The above described heat machine can operate as a refrigerator as long as the heat extracted from the cold reservoir during the fourth step is nonnegative. This heat is equal to the difference between the average energies of states $A$ and $D$, thus
\begin{equation}
\label{ref_condition}
Q_4=E_A-E_D\geq 0\Rightarrow E_A\geq E_D.
\end{equation}
If the frequency of the harmonic oscillator is restricted as
\begin{equation*}
\omega_c\leq\omega(t)\leq\omega_h,
\end{equation*}
then the minimum value of $E_D$ which can be obtained at the end of the third step, starting from the equilibrium value $E_C$, is
\begin{equation}
\label{EDmin}
E_{D,min}=\frac{\omega_c}{\omega_h}E_C,
\end{equation}
according to (\ref{Emin}) with the analogy
\begin{equation}
\label{analogy}
\omega_h\rightarrow\omega_0,\quad\omega_c\rightarrow\omega_f,\quad E_C\rightarrow E_0,\quad E_D\rightarrow E_f.
\end{equation}
As explained in section \ref{sec:formulation}, this minimum value can be achieved if the available time $T$ for the third step of the cycle is larger than a necessary minimum time $T_{min}$, which can be calculated following the procedure described in our recent work \cite{Stefanatos16a}.
The following proposition provides the ordering of $E_A, E_C, E_{D,min}$.
\begin{proposition}
\label{energies}
If $\omega_h/\omega_c>\mathcal{T}_h/\mathcal{T}_c$, then $E_{D,min}<E_A<E_C$
\end{proposition}

\begin{proof}
$E_A, E_C$ are average energies of thermal equilibrium states at temperatures $\mathcal{T}_c, \mathcal{T}_h$ and frequencies $\omega_c, \omega_h$, respectively, thus
\begin{equation}
\label{Equilibrium_Energies}
E_A=\frac{\hbar\omega_c}{2}\coth\left(\frac{\hbar\omega_c}{2k_b\mathcal{T}_c}\right),\quad E_C=\frac{\hbar\omega_h}{2}\coth\left(\frac{\hbar\omega_h}{2k_b\mathcal{T}_h}\right),
\end{equation}
where $k_b$ is Boltzmann's constant. Using these expressions and (\ref{EDmin}) we obtain
\begin{equation*}
E_A>E_{D,min}\Leftrightarrow\coth\left(\frac{\hbar\omega_c}{2k_b\mathcal{T}_c}\right)>\coth\left(\frac{\hbar\omega_h}{2k_b\mathcal{T}_h}\right),
\end{equation*}
which is true since $\coth$ is a decreasing function of its argument and $\omega_h/\omega_c>\mathcal{T}_h/\mathcal{T}_c\Rightarrow\omega_c/\mathcal{T}_c<\omega_h/\mathcal{T}_h$. For the second inequality $E_C>E_A$ note that, if we set
\begin{equation*}
\label{ratios}
\frac{\hbar\omega_c}{2k_b\mathcal{T}_c}=x,\quad\frac{\hbar\omega_h}{2k_b\mathcal{T}_h}=y,\quad\frac{\mathcal{T}_c}{\mathcal{T}_h}=a,
\end{equation*}
then it becomes
\begin{equation*}
y\coth y>ax\coth x
\end{equation*}
Since the temperature of the hot reservoir is obviously larger than that of the cold reservoir, it is $a<1$, thus it is sufficient to show that $y\coth y>x\coth x$. But $y/x=\omega_h\mathcal{T}_c/\omega_c\mathcal{T}_h>1$, so it is sufficient to show that the function $h(z)=z\coth z$ is increasing for $z\geq 0$. It is
\begin{equation*}
h'(z)=\frac{\sinh(2z)-2z}{2\sinh^2 z}
\end{equation*}
and if we set $w(z)=\sinh(2z)-2z$ it is also
\begin{equation*}
w'(z)=2[\cosh(2z)-1]\geq 0.
\end{equation*}
Thus $w(z)\geq w(0)=0$ and $h'(z)\geq h'(0)=0$ for $z\geq 0$, so $h(z)$ is indeed an increasing function of its argument.
\end{proof}

On the other hand, if the third step of the cycle is a sudden quench where $T=0$, the corresponding energy is
\begin{equation*}
E_{D,sc}=\left(1+\frac{\omega_c^2}{\omega_h^2}\right)\frac{E_C}{2},
\end{equation*}
which is obtained from (\ref{sudden_quench}) using the analogy (\ref{analogy}). Unlike to the previous case, there is no constant ordering between $E_A, E_{D,sc}$ but it depends on the values of the parameters, frequencies and temperatures. If $E_{D,sc}<E_A$, then the minimum driving time for the third step of the cycle is obviously $T=0$; the heat machine can operate as a refrigerator even if the third step is a sudden quench. The interesting situation arises when $E_{D,sc}\geq E_A$. In this case, there is a minimum driving time for the third step of the cycle; if this step is performed faster, then the heat machine ceases to operate as a refrigerator. According to (\ref{ref_condition}), this minimum driving time is encountered when $E_D=E_A$. In this case we have
\begin{equation}
\label{rE}
r_E=\frac{E_C}{E_D}=\frac{E_C}{E_A}>1,\quad\gamma=\sqrt{\frac{\omega_h}{\omega_c}}>1,
\end{equation}
following the analogy (\ref{analogy}). With these values for $\gamma$ and $r_E$, the minimal driving time can be calculated using Theorem \ref{prop:time}. Note that, as pointed out in \cite{Abah16}, the calculation of this time is a difficult task, but here we provide the appropriate framework and a systematic procedure.

As an example we consider the case with $\gamma=10\Rightarrow\omega_h/\omega_c=\gamma^2=100$, where the energy ratio is taken such that the final curve (\ref{final_curve}) meets the $x_1$-axis at the point $(\bar{\gamma}=8,0)$. This is done for comparison reasons with the case where the final point is $(\bar{\gamma}=8,0)$, as we discuss below. If we set $x_1=\bar{\gamma}=8, x_2=0$ in (\ref{final_curve}) we obtain
\begin{equation*}
r_E=\frac{2\bar{\gamma}^2}{1+\left(\frac{\bar{\gamma}}{\gamma}\right)^4}=90.8059.
\end{equation*}
If we also set
\begin{equation*}
\mathcal{T}_c=\frac{\hbar\omega_c}{2k_b}=\frac{1}{\gamma^2}\frac{\hbar\omega_h}{2k_b}=0.01\frac{\hbar\omega_h}{2k_b}
\end{equation*}
then, from (\ref{rE}) and (\ref{Equilibrium_Energies}) we obtain
\begin{equation*}
\mathcal{T}_h=\frac{1}{\coth^{-1}\left[\frac{r_E}{\gamma^2}\coth\left(\frac{\hbar\omega_c}{2k_b\mathcal{T}_c}\right)\right]}\frac{\hbar\omega_h}{2k_b}=0.8218\frac{\hbar\omega_h}{2k_b}.
\end{equation*}
\begin{table}[tbhp]
\caption{Extremal times (units $\omega_h^{-1}$)}
\label{tab:times}
\centering
\begin{tabular}{|c|c|c|c|c|} \hline
$T^+_{11}$ & $T^+_{12}$ & $T^+_{31}$ & $T^+_{32}$ &  $T^+_{51}$ \\ \hline
8.00794 & 12.53205 & \textbf{7.38567} & 8.77552 & 9.55663 \\  \hline \hline
$T^+_{52}$ & $T^-_{31}$ & $T^-_{32}$ & $T^-_{51}$ & $T^-_{52}$ \\ \hline
10.49350 & 9.76875 & 14.22294 & 9.57303 & 10.80736 \\ \hline
\end{tabular}
\end{table}
In Table \ref{tab:times} we show the times (in units $\omega_h^{-1}$), calculated using Theorem \ref{prop:time} with the above parameter values, corresponding to the extremals for which the transcendental equation (\ref{transcendental}) has at least one solution. In each of these times, the superscript indicates which ($\pm$) transcendental equation was used, while in the subscript the first number indicates the switchings and the second one the order of the solution. The minimum time is highlighted with bold, while the corresponding optimal trajectory is depicted in Fig. \ref{fig:trajectory}. Blue solid line corresponds to $X$-segments ($u=u_1$), while red dashed line corresponds to $Y$-segments ($u=u_2$). The black solid line close to the $x_1$-axis corresponds to the final curve. Observe that the final point of the trajectory lies close to the point $(\bar{\gamma}=8,0)$, where the final curve meets $x_1$-axis. The situation is magnified in Fig. \ref{fig:detail}. The minimum time $T^+_{31}=7.38567$ to reach the final curve is slightly smaller than the minimum time to reach $(\bar{\gamma}=8,0)$, $\bar{T}=7.38568$, which can be calculated using the results of \cite{Stefanatos16a}. As it is clear from Fig. \ref{fig:detail}, the steep slope of the final curve close to $(8,0)$ is exploited to obtain a lower minimum time.
\begin{figure}[tbhp]
\centering
\subfloat[Optimal trajectory]{\label{fig:trajectory}\includegraphics[width=0.5\linewidth]{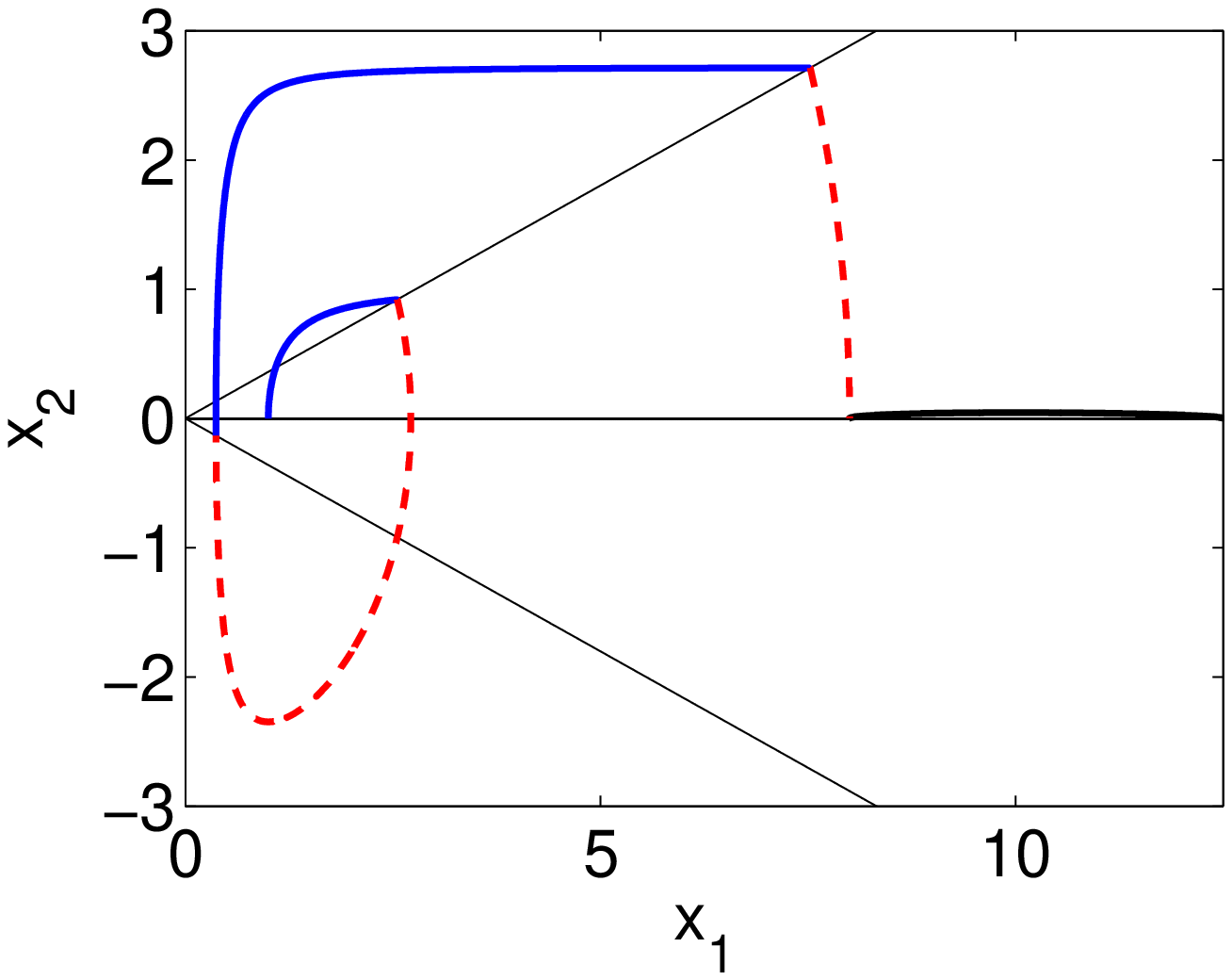}}
\subfloat[Magnified detail]{\label{fig:detail}\includegraphics[width=0.5\linewidth]{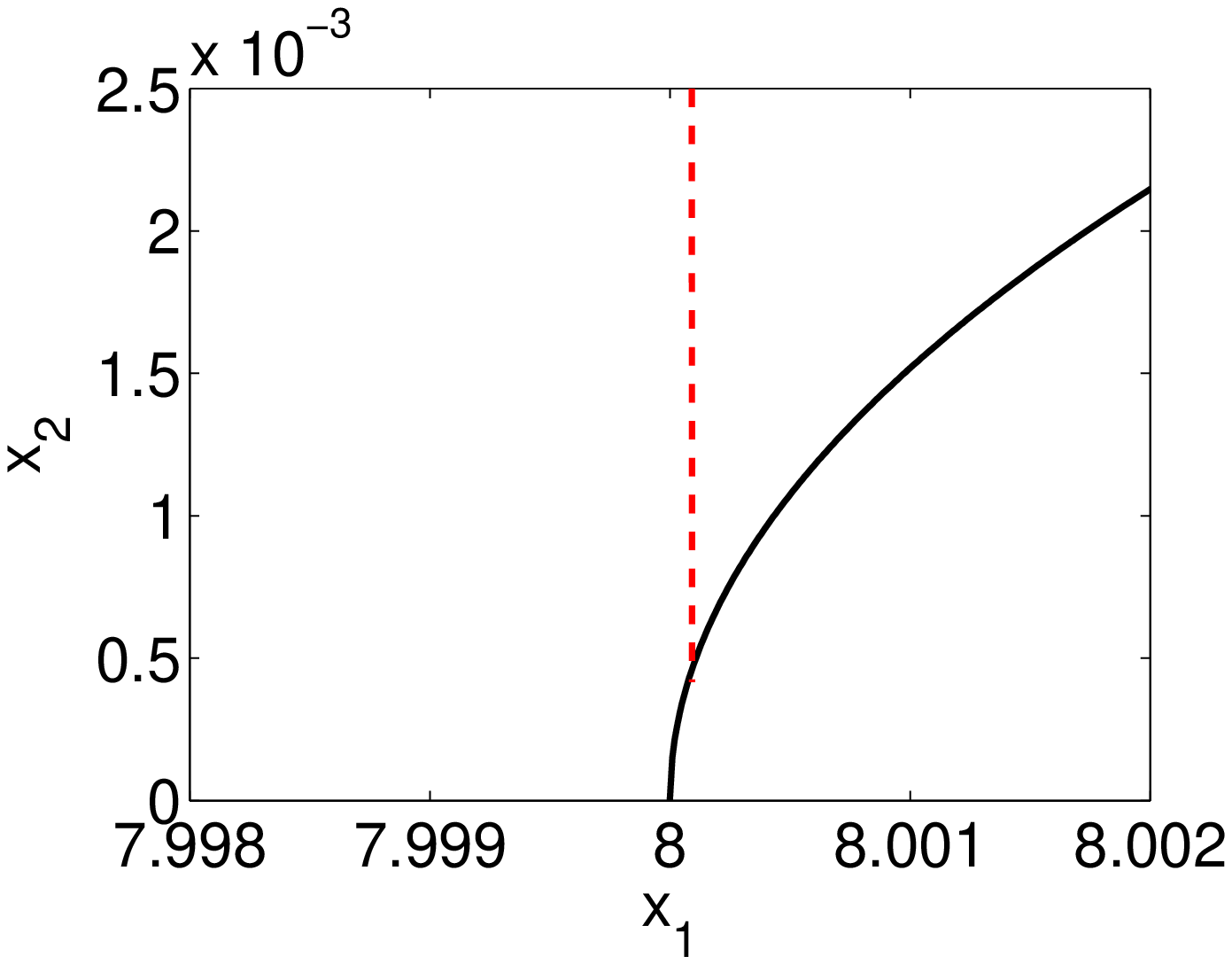}}
\caption{Optimal trajectory for the parameter values given in the text (left panel). Magnified detail around the final point on the target curve (right panel).}
\label{fig:optimal_trajectory}
\end{figure}

\subsection{Connection with quantum finite-time availability}

The concept of quantum finite-time availability describes the potential work which can be obtained by a finite-time process which is too short to gain all the work by bringing a thermodynamic ensemble of quantum systems into thermal equilibrium with an environment \cite{Hoffmann15}. Consider for example the third step of the Otto cycle above, where the frequency of the oscillator is decreased from $\omega_h$ to $\omega_c$. This isentropic expansion (recall that the frequency corresponds to inverse volume) is analogous to the expansion of a piston, thus work is performed. If the expansion time $T$ is larger than a necessary minimum time $T_{min}$, which can be calculated following the procedure described in our recent work \cite{Stefanatos16a}, then the available work takes its maximum value which is
\begin{equation*}
W_{max}=E_C-E_{D,min}=\left(1-\frac{\omega_c}{\omega_h}\right)E_C.
\end{equation*}
If $T<T_{min}$ then $\mbox{min}\{E_D\}>E_{D,min}$ and the available work is
\begin{equation*}
W=E_C-\mbox{min}\{E_D\}<W_{max}
\end{equation*}
The authors of \cite{Hoffmann15} consider an extremal of the form $XY$ with only one intermediate switching, fix the time $T<T_{min}$ and obtain the available work $W$ by minimizing numerically $E_D$. The framework presented in the present paper can obviously be used to solve the dual problem: Fix $E_D>E_{D,min}$ and find the corresponding minimum time $T<T_{min}$. This framework is more general, since it can provide more complex solutions, like for example the $XYXY$ optimal trajectory of the previous subsection.

We close the applications section by using the formulas of Theorem \ref{prop:time} to elucidate a point made by the authors of \cite{Hoffmann15}. Specifically, they numerically observe that for an $XY$ extremal, when the process time $T\rightarrow 0$, the times spent on the $X$- and $Y$-segments tend to be equal. Here we show that this is indeed the case. Note first that for an $XY$ extremal with duration $T\rightarrow 0$, the time spent on the $X$-segment is $\tau_-$ (\ref{time_in}), corresponding to the switching point closer to the starting point, while the time spent on the $Y$-segment is $\tau$ (\ref{time_fi}). The limit $T=\tau_-+\tau\rightarrow 0$ corresponds to $s\rightarrow 0$, since in this limit we have $\cos(2\sqrt{u_1}\tau_-)\rightarrow 1$ and $\cos(2\sqrt{u_2}\tau)\rightarrow 1$ from (\ref{time_in}), (\ref{time_fi}). If we use the same equations to expand these cosines to first order in $s$ we obtain
\begin{align*}
\cos(2\sqrt{u_1}\tau_-)&=1-\frac{2u_1}{(1-u_1)^2}s+h.o.t.,\\
\cos(2\tau)&=1-\frac{2}{(1-u_1)^2}s+h.o.t.,
\end{align*}
where we have used that $u_2=1$. Using the small $x$ expansion $\cos x\approx 1-x^2/2$, we obtain
\begin{equation}
\label{equal_times}
\tau_-\approx\tau\approx\frac{\sqrt{s}}{1-u_1}.
\end{equation}

Consider for example the case with $\gamma=10$ ($\omega_h/\omega_c=\gamma^2=100$) and
\begin{equation*}
r_E=\frac{2\gamma^4}{\gamma^4+1}+0.0005=2.0003
\end{equation*}
\begin{figure}[htbp]
  \centering
  \includegraphics[width=0.5\linewidth]{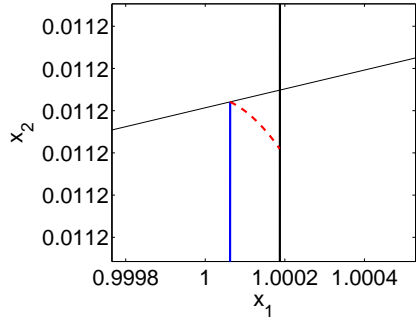}
  \caption{$XY$ extremal when $r_E$ is very close to its minimum value.}
  \label{fig:small_s}
\end{figure}
where the energy ratio is taken slightly larger than its minimum value given in (\ref{r}). As a consequence of this choice of $r_E$, the final curve lies very close to the $X$-segment. This can be seen from the axis numbering in Fig. \ref{fig:small_s} where we have done a substantial magnification to distinguish between the $X$-segment (blue solid line) and the final curve (black solid line). Note that the $Y$-segment (red dashed line) is also shown. Using the transcendental equation (\ref{transcendental}) with $n=0$ we find $s=0.000125$ and duration $T^-_1=\tau_-+\tau=0.022364$, with $\tau_-=0.011183$ and $\tau=0.011181$. The approximate formula (\ref{equal_times}) gives $\tau_-\approx\tau\approx 0.011181$, in very good agreement with the numerically obtained values.

\section{Conclusions}
\label{sec:conclusions}

Using the tools of geometric optimal control, we solved the problem of minimum-time transitions between thermal equilibrium and fixed average energy states of the quantum parametric oscillator. We then applied the results obtained to answer two questions from quantum thermodynamics. First, to find the minimum driving time for a quantum refrigerator, and second, to quantify the quantum finite-time availability of the parametric oscillator.



\bibliographystyle{plain}
\bibliography{shortbib}




\end{document}